\documentclass[12pt]{amsart}

\theoremstyle{definition}
\newtheorem{definition}{Definition}
\theoremstyle{plain}
\newtheorem{theorem}{Theorem}
\newtheorem{corollary}{Corollary}
\newtheorem{lem}{Lemma}
\newtheorem{prop}{Proposition}
\newtheorem{ejem}{Example}
\theoremstyle{remark}

\newtheorem{remark}{Remark}
\usepackage{color}

\DeclareMathOperator{\CC}{\mathbb C}
\DeclareMathOperator{\RR}{\mathbb R}

\DeclareMathOperator{\NN}{\mathbb N}

 \DeclareMathOperator{\sop}{supp}

\DeclareMathOperator{\essupp}{essupp}

\DeclareMathOperator{\Co}{Co}
\DeclareMathOperator{\Pc}{Pc}

\title{Smallest and largest generalized eigenvalues of large  moment  matrices and some applications.
 }

\author[1,2]{C. Escribano}

\author[1]{R. Gonzalo}

\author[1]{E. Torrano}

\address{ Departamento de Matem\'atica Aplicada a las Tecnolog\'{i}as de la Informaci\'on y las Comunicaciones, Escuela T\'ecnica Superior de Ingenieros Inform\'aticos y Center for Computational Simulation, Universidad Polit\'ecnica de Madrid, Campus de Montegancedo\\
      Boadilla del Monte, 28660 Madrid Spain, Phone: +34910673030 \\
      }

\email{cescribano@fi.upm.es}
\email{rngonzalo@fi.upm.es}
\email{emilio@fi.upm.es}
\thanks{Work partially supported by Comunidad
de Madrid (grant QUITEMAD-CM, ref. P2018/TCS-4342).}

\begin{document}

\begin{abstract}
The main aim of this work is to compare two Borel measures thorough their moment matrices using a new notion of smallest and largest generalized eigenvalues. 
 With this  approach we provide information in  problems as the localization of the support of a measure. In particular, we  prove that if a measure is {\it comparable} in an algebraic way with a measure in a Jordan curve then the curve is contained in its support. We  obtain a description of the convex envelope of the support of a measure via certain  Rayleigh quotients of certain infinite matrices. Finally some applications concerning  polynomial approximation  in  mean  square  are given, generalizing the results in \cite{EGT1}.

\end{abstract}

\maketitle
\begin{quotation} {\sc {\footnotesize Keywords}}. {\small Hermitian moment problem,  orthogonal polynomials,
smallest eigenvalue, measures, approximation by polynomials}
\end{quotation}

\section{Introduction}

\noindent  Let $\mathbf{M}=(c_{i,j})_{i,j=0}^{\infty}$ be an infinite 
Hermitian matrix, i.e., $c_{i,j}= \overline{c_{j,i}}$ for all
$i,j$ non-negative integers. Following \cite{EGT1} we say that an infinite hermitian matrix $\mathbf{M}$ is positive definite, in short an HPD matrix (resp. positive semidefinite, in short HSPD matrix)   if
$\vert \mathbf{M}_n \vert >0$ (resp. $\vert \mathbf{M}_n \vert \geq 0$ for all $n\geq 0$), where $\mathbf{M}_n$ is the truncated matrix of size $(n+1)\times (n+1)$ of $\mathbf{M}$. Following the notation in \cite{EGT1} and \cite{EGT5} an  infinite
HPD  matrix defines an inner product, denoted by  $\langle \cdot, \cdot \rangle_{\mathbf{M}}$,  in the space $\mathbb{P}[z]$ of all polynomials with complex coefficientes in the following way:
if  $p(z)=\sum_{k=0}^{n}v_kz^k$ y
$q(z)=\sum_{k=0}^{m}w_kz^k$ then
\begin{equation}
 \label{uno}
\langle p(z)
,q(z) \rangle_{\mathbf{M}}=v\mathbf{M}w^{*}
\end{equation}

\noindent being $v=(v_0,\dots,v_n,0,0, \dots), w=(w_0,\dots,w_m,0,0, \dots) \in c_{00}$ where $c_{00}$ is the space
of all complex sequences with only finitely many non-zero entries and the 
symbol $*$ denotes the conjugate transpose for matrices of any sice.
  The associated
norm is denoted by $\Vert p(z) \Vert_{\mathbf{M}}^2 = \langle p(z),p(z)\rangle_{\mathbf{M}} $ for every $p(z)\in \mathbb{P}[z]$, and the vector space $\mathbb{P}[z]$ is a pre-hilbertian space endowed with such norm and its completion will be  denoted by $P^{2}_{\mathbf{M}}$. In this way,  $P^{2}_{\mathbf{M}}$ is a Hilbert space associated to the inner product  $\langle \cdot ,\cdot \rangle_{\mathbf{M}}$.

\medskip

\noindent Of course, the most  interesting
class of infinite HPD (and also HSPD) matrices are those which are  moment matrices
 with respect to a certain positive measure (see e.g. \cite{Van Assche}),
i.e.,  matrices  $\mathbf{M}=(c_{i,j})_{i,j=0}^{\infty}$ such that there exists a  positive Borel measure $\mu$ with support in the complex plane, which is called a {\it representating measure} of $\mathbf{M}$,  with  finite moments for all $i,j\geq 0$ 
\begin{equation} \label{dos}
c_{i,j}= \int z^{i} \overline{z}^jd\mu.
\end{equation}
\noindent Thorough the paper we always consider positive  Borel measures 
$\mu$ supported in the complex plane and with finite moments. The associated moment matrix to a measure $\mu$, denoted by $\mathbf{M}(\mu)$,  is always an infinite HSPD, being  an  HPD matrix  if and only if the support of $\mu$ is an infinite set.  For more information concerning the characterization of infinite HPD matrices which are moment matrices with respect to a certain measure $\mu$ with support in  $\CC$ see among others  \cite{Atzmon}, \cite{Berg-Maserik} and \cite{Szafraniec}.

\bigskip

\noindent Our aim here is to  study certain geometrical and localization  properties of the support of measures  thorough their associated moment matrices.  This will be done   in  the context of matrix algebra  using as  essential tool  the infinite matrices and the generalizations of some well  known notions as generalized eigenvalues. 
The  approach is the same as in  \cite{Szego}, \cite{Berg}, \cite{EGT1} and \cite{EGT5} and the essential key 
is  the following identity: if $p(z)=v_0+v_1z+\dots+v_nz^n \in \mathbb{P}[z]$ and $(v_0, v_1, \dots,v_n,0,0,0,\dots)\in c_{00}$,
\begin{equation} \label{identification}
v\mathbf{M}(\mu)v^{*}=\int \vert p(z) \vert^2 \; d\mu . \qquad \qquad \qquad  
\end{equation}

\noindent In the case of Toeptliz definite positive matrices, which are moment matrices associated to measures with support in the unit circle, the behaviour  of various spectral characteristics, in particular eigenvalues, eigenvectors, and extreme eigenvalues, has been an object of active investigation (see e.g. the classical  Szego's paper  \cite{Szego} or \cite{Grenader-Szego}). 
In  the case of Hankel positive moment  matrices in  \cite{Berg-Chen-Ismail} a characterization of the uniqueness of the representing measure  is given in terms of the asymptotic behaviour of the smallest eigenvalues of the truncated matrices  of the moment matrix. The analysis of the largest and smallest eigenvalues plays an important role since they provide useful information about the nature of Hankel matrices generated for weight functions and is a topic of great interest and activity (see e.g. \cite{Berg}, \cite{Berg1}, and \cite{Wang}  among others).  We here are concerned with   general    moment matrices associated with measures supported  in the complex plane. In this direction  in
\cite{EGT1}  a sufficient condition is given to assure  polynomial approximation in the corresponding space of measure $L^{2}({\mu})$ (for compactly supported measures ) also in terms of the asymptotic behaviour of the smallest eigenvalues. With  the same approach in  \cite{EGT5}, for  measures supported in Jordan curves,   a characterization of polynomial approximation is given. 

\bigskip

\noindent The main novelty in this work is the analysis of   {\it generalized} eigenvalues of large hermitian matrices  as a new tool in order to compare two different measures or even sets ( suggested in  \cite{Szego}). Our motivation was to realize that  the analysis  of  eigenvalues   can be seen as a way to compare any infinite hermitian  matrix with the  identity matrix $\mathbf{I}$, which is indeed the 
moment matrix associated with the uniform
 Lebesgue measure  in the unit circle. Such comparison   does not provide, in general,  any information concerning density polynomial when we consider
measures supported in other  disks. Nevertheless, we can  make such comparison between any two measures via {\it generalized eigenvalues} as we will do; with this approach we may   describe the support of one measure  in relation with the support of another.

\bigskip 
\noindent In the first section, given two infinite HPD matrices we introduce two indexes that describe the assymptotic behaviour of the   smallest and largest generalized eigenvalues of large hermitian   matrices. In the particular case of moment matrices  these indexes provide an algebraic tool to compare measures. Some properties of these indexes will be given. 

\bigskip
\noindent In the second section we see the impact of the largest generalized eigenvalue of a measure  with respect to other in the support of the first in relation with the support of the second. The main result of this  section asserts that if such largest generalized eigenvalue is finite then the support of the first measure is contained in the polynomially convex hull of the support of the second. In particular, for measures with polynomially convex support, that means without {\it holes} in their supports, that are {\it well comparable} between them (with positive   smallest generalized eigenvalue  and finite  largest generalized eigenvalue ) the supports coincide. In the particular case of a good comparison of a 
certain measure with a measure supported in  a Jordan curve we obtain that the Jordan  curve must be contained in the support of such measure; this lets to find Jordan curves in the supports of a measure using this algebraic approach. 

\bigskip
\noindent The third section is devoted to the description of the convex envelope of the support of a measure. Our main result is that the convex hull of the support of a measure can be obtained  using Rayleigh quotients of a certain infinite matrix (mainly, the moment matrix without the first row) with the moment matrix. In order to prove this result we use theory of normal operators in Hilbert spaces and results concerning Hessemberg matrices in the theory of orthogonal polynomials. 

\bigskip
\noindent In the last section we  generalize the results of polynomial approximation in mean square in \cite{EGT1}. This generalization is obtained  when we replace the uniform Lebesgue measures for other uniform Lebesgue measures in circles, which are, in particular, non complete . We do not know if we can replace these measures by any  measure for which there is no  polynomial approximation in mean square. In this direction, we give a partial answer when the first measure is supported in a Jordan curve. 

\bigskip

\noindent We begin with some notation and definitions:

\bigskip
\noindent Thorough all the paper  we denote by $\mathbb{D}=\{z\in \CC\; \vert z \vert <1\}$ and $\mathbb{T}=\{z\in \CC\; \vert z \vert=1\}$. For general disks we denote $\mathbb{D}_{r}(z_0)=\{ z\in \CC : \; \vert z-z_0 \vert <r\}$ and $\mathbb{S}_{r}(z_0)= \{ z\in \CC : \; \vert z-z_0 \vert =r\}$. We denote by ${\bf m}$ the uniform Lebesgue measure in the unit sphere; for general spheres we denote by 
${\bf m}_{z_0;r}$ the uniform Lebesgue measure in the sphere $\mathbb{S}_{r}(z_0)$ and by simplicity in the particular case of spheres with center $0$ we denote  ${\bf m}_r={\bf m}_{0,r}$. 

\section{Generalized eigenvalues associated to infinite SHPD matrices.}

\noindent Our aim in this section is to extend  the notion of generalized eigenvalues in the context of infinite dimensional  matrices. Recall that (see e.g.  \cite{Bhatia})

\begin{definition} Given two finite matrices of size $n\times n$,  $\mathbf{A}$ 
and $\mathbf{B}$,  we say that $\lambda$ is an eigenvalue of $\mathbf{A}$ with respect to $\mathbf{B}$ if there exists a nonzero vector $v\in \CC^{n}$ such that $\mathbf{A}v= \lambda \mathbf{B}v$. These are called the {\it generalized eigenvalues} of $\mathbf{A}$ with respect $\mathbf{B}$ .
\end{definition}

\noindent It is well known that in the
 case of hermitian matrices the generalized eigenvalues are real numbers. 
This lets to  define the following indexes:

\begin{definition} Let $\mathbf{A},\mathbf{B}$ be Hermitian matrices of size $n$, we define $\lambda_n(\mathbf{A},\mathbf{B})$ as the smallest eigenvalue of $\mathbf{A}$ with respect to $\mathbf{B}$ and
$\beta_n(\mathbf{A},\mathbf{B})$ as the largest eigenvalue of $\mathbf{A}$ with respect to $\mathbf{B}$.
\end{definition}

\medskip

\noindent It is obvious that $\lambda_n(\mathbf{A},\mathbf{B}) \leq \beta_n(\mathbf{A},\mathbf{B})$. Moreover, in the case of Hermitian semi-definite positive matrices  $\lambda_n(\mathbf{A},\mathbf{B}) \geq 0$.

\bigskip

\noindent  Of course, the notion of generalized eigenvalue is a generalization of the notion of eigenvalues. Indeed, in  the particular case of $\mathbf{B}=\mathbf{I}$ (the identity matrix of size n ) it is obvious that  $\lambda_n(\mathbf{A},\mathbf{I})=\lambda_n(\mathbf{A})$ and  $\beta_n(\mathbf{A},\mathbf{I})=\beta_n(\mathbf{A})$ with the notation introduced in \cite{EGT1}.
It is well known by the classical Rayleigh quotient that for finite hermitian matrices it follows 

\begin{eqnarray*}
\lambda_{n}(\mathbf{A})& =& \inf\{
 \dfrac{ v\mathbf{A}v^{*}}{vv^{*}}, v\in \CC^{n}\setminus\{0\}\}= \inf\{ v\mathbf{A}v^{*}: \; v\in \CC^{n},  vv^{*}=1\},\\
\beta_n(\mathbf{A}) & =& \sup \{ \dfrac{ v\mathbf{A}v^{*}}{vv^{*}} :  v\in \CC^{n}\setminus\{0\}\}   =  \sup \{v\mathbf{A}v^{*}:  \; v\in \CC^{n}, vv^{*}=1\}.
\end{eqnarray*}

\bigskip

\noindent This criteria also holds for the generalized eigenvalues replacing the identity matrix by an HPD matrix $\mathbf{B}$ ( see e.g.  \cite{Gantmacher}). 

\bigskip

\noindent ({\it Generalized  Rayleigh  criteria}) Let $\mathbf{A}, \mathbf{B}$ be HSPD matrices of size  $n$ being $\mathbf{B}$ positive definite, then:

\begin{eqnarray*}
\lambda_{n}(\mathbf{A},\mathbf{B}) & = & \inf \{ v\mathbf{A}v^{*}: \;v\in \CC^{n},   v\mathbf{B}v^{*}=1\}=\inf\{ \dfrac{ v\mathbf{A}v^{*}}{v\mathbf{B}v^{*}} :   v\in\CC^{n}\setminus\{0\} \; \},\\
\beta_n(\mathbf{A},\mathbf{B}) & =& \sup \{v\mathbf{A}v^{*}: \; v\in \CC^{n},  v\mathbf{B}v^{*}=1\}=\sup \{ \dfrac{ v\mathbf{A}v^{*}}{v\mathbf{B}v^{*}} :  v\in \CC^{n}\setminus\{0\} \; \}.
\end{eqnarray*}

\bigskip
\noindent The definition of the Rayleigh quotients can be given in the more general context of infinite matrices (not necessarily hermitian ) in the following way:

\begin{definition}  \label{definicion4}
Given an  infinite matrix $\mathbf{A}$ and an infinite HPD  matrix $\mathbf{B}$ a  Rayleigh quotient of $\mathbf{A}$ related to $\mathbf{B}$ is $\dfrac{v\mathbf{A}v^{*}}{v\mathbf{B}v^{*}}$ with $v \in c_{00}\setminus \{0\}$ and the set of Rayleigh quotients is the following set in the complex plane:
$$
\{\; \dfrac{v\mathbf{A}v^{*}}{v\mathbf{B}v^{*}}: \;   v\in  c_{00} \setminus\{0\} \; \}=\{v\mathbf{A}v^{*} : \;  v\in c_{00}, \; v\mathbf{B}v^{*}=1\}.
$$
\end{definition}

\bigskip

\noindent  We generalize the indexes $\lambda$ and $\beta$ introduced for 
finite matrices  in the context of infinite dimensional matrices in the same lines as in \cite{EGT1}, \cite{EGT5}.
In order to do it, consider  two HSPD matrices  $\mathbf{A}, \mathbf{B}$ with $\mathbf{B}$ being positive definite, and let $\mathbf{A}_n, \mathbf{B}_n$ the their truncated matrices of size $(n+1)\times (n+1)$. It is easy to check that the sequence of real numbers  $\{\lambda_{n}(\mathbf{A}_n,\mathbf{B}_n)\}_{n=0}^{\infty}$ is a non increasing sequence and the sequence $\{\beta_n(\mathbf{A}_n,\mathbf{B}_n)\}_{n=0}^{\infty}$ is a non decreasing sequence. Therefore, $\displaystyle{\lim_{n\to
\infty}\lambda_{n}(\mathbf{A}_n,\mathbf{B}_n)}$  exists, and either is zero or a positive number. On the other hand, $\displaystyle{\lim_{n\to
\infty}\beta_{n}(\mathbf{A}_n,\mathbf{B}_n)}$ exists (is a real number) or it is infinite. Thus, we introduce the following indexes

\bigskip

\begin{definition} Let  $\mathbf{A},\mathbf{B}$ two infinite HSPD matrices with $\mathbf{B}>0$. We define:
\begin{eqnarray*}
0 \leq \lambda(\mathbf{A},\mathbf{B})&:=&\lim_{n\to \infty} \lambda_{n}(\mathbf{A}_n,\mathbf{B}_n)\\
\beta(\mathbf{A},\mathbf{B})& := &  \lim_{n\to
\infty}\beta_{n}(\mathbf{A}_n,\mathbf{B}_n) \leq \infty
\end{eqnarray*}
\noindent where $\mathbf{A}_n, \mathbf{B}_n$ are the corresponding truncated matrices of size  $(n+1)\times (n+1)$ of $\mathbf{A},\mathbf{B}$ respectively.

\end{definition}

\bigskip

\noindent Moreover, using the generalized Rayleigh criteria it follows:

\begin{eqnarray*}
\lambda(\mathbf{A},\mathbf{B})& = & \inf \{\; \dfrac{v\mathbf{A}v^{*}}{v\mathbf{B}v^{*}}: \;  v\in  c_{00} \setminus \{0\} \; \}=\inf \{v\mathbf{A}v^{*} : \;  v\in c_{00}, \; v\mathbf{B}v^{*}=1\},\\
\beta(\mathbf{A},\mathbf{B}) & = &  \sup\{\;  \dfrac{v\mathbf{A}v^{*}}{v\mathbf{B}v^{*}}: \;   v\in  c_{00} \setminus \{0\} \; \}
= \sup \{v\mathbf{A}v^{*} : \;  v\in c_{00}, \; v\mathbf{B}v^{*}=1\}.
\end{eqnarray*}

\bigskip

\begin{remark} In the case that $\mathbf{A}>0$ and $\mathbf{B}>0$ the following relationship between these indexes is given:
\begin{equation} \label{tres}
\lambda(\mathbf{A},\mathbf{B})=\frac{1}
{\beta(\mathbf{B},\mathbf{A})}
\end{equation}

\end{remark}

\bigskip

\smallskip
\noindent The particularization of the indexes $\lambda$ and $\beta$ for infinite moment matrices lets to  introduce  these  indexes between measures in the following way:

 \begin{definition} \label{definicion7}
 Given  two   supported in the complex 
plane measures  $\mu_{1},\mu_{2}$ with $\mu_{2}$ infinitely  supported we define the following indexes 

\[
\lambda(\mu_1,\mu_2):= \lambda\left(\mathbf{M}(\mu_1),\mathbf{M}(\mu_2)\right) \qquad \beta(\mu_1,\mu_2):=\beta(\mathbf{M}(\mu_1),\mathbf{M}(\mu_2)).
\]

\end{definition}

\bigskip

\begin{remark}  By remark $1$ in the case that both measures $\mu_1,\mu_2$ have infinite support the following relationship between the indexes follows:
$$
\lambda(\mu_1, \mu_2)=\frac{1}{\beta(\mu_2,\mu_1)}.
$$
\end{remark}
\bigskip

\begin{ejem}  \label{ejemplo1}  For uniform Lebesgue measures ${\bf m}_r$ with $r>0$ it holds 

\begin{enumerate}
\item[i)] If $r<1$,
$
\lambda({\bf m_{r}},{\bf m})=0, \qquad \beta({\bf m_{r}},{\bf m})=1.
$
\medskip

\item[ii)] If $r>1$,
$
\lambda({\bf m_{r}},{\bf m})=1, \qquad \beta({\bf m_{r}},{\bf m})=\infty .
$
\end{enumerate}

\noindent Indeed, denote by  ${\bf M}_{r}$  the associated moment matrix with ${\bf m}_r$ which is given by:

$$
\mathbf{M}_{r}= \left(
         \begin{array}{cccc}
           1 & 0 & 0 & \hdots \\
           0 & r^2 & 0 & \hdots \\
           0 & 0 & r^4 & \hdots \\
           \vdots & \vdots & \vdots & \ddots  \\
         \end{array}
       \right).
$$
\noindent Then the truncated matrices of size $(n+1)\times(n+1)$ are diagonals and the result follows easily.

\end{ejem}

\bigskip

\begin{ejem} \label{ejemplo2}
The Pascal matrix $\mathbf{P}$ is the moment matrix   associated with the normalized Lebesgue measure ${\bf m}_{1,1}$ in $\mathbb{S}_{1}(1)$ ( see e.g. \cite{EGT1})
$$
\mathbf{P}= \left(
              \begin{array}{ccccc}
                1 & 1 & 1 & 1 & \hdots  \\
                1 & 2 & 3 & 4 & \hdots \\
                1 & 3 & 6 & 10 & \hdots \\
                1 & 4 & 10 & 20 & \hdots \\
                \vdots  & \vdots & \vdots & \vdots & \ddots \\
              \end{array}
            \right).
$$

\noindent As  it was pointed out in \cite{EGT1} $\displaystyle{\lim_{n\to 
\infty} \lambda_n(\mathbf{P})=\lambda({\bf m}_{1,1},{\bf m})=0}$. On the other hand, $\beta({\bf m}_{1,1},{\bf m})=\infty$; indeed,  by Rayleigh criterion $\beta(\mu,{\bf m}) =\beta(\mathbf{P},\mathbf{I}) \geq e_n
\mathbf{P}e_n^{*}=c_{n,n} \to \infty$ whenever  $n\to \infty$ (where $e_n$ is the usual basis in the space  $c_{00}$).
\end{ejem}

\noindent The following result, which is a generalization  of the results in \cite{EGT1},  is the key to establish the  comparison of measures via generalized eigenvalues of the associated moment matrices:

\begin{prop} \label{proposicion1}
Let  $\mu_{1},\mu_{2}$ be  two   measures  supported in the complex plane  with  $\mu_2$ infinitely supported and let $\mathbf{M}_i=\mathbf{M}(\mu_{i}),i=1,2,$   their associated  moment matrices. Then, 

\begin{itemize}
\item[(i)] $\beta(\mu_1,\mu_2)<\infty$ if and only if there exists 
 a constant $c>0$ such that for every polynomial $p(z)$
$$
\int \vert p(z) \vert^2 d\mu_1 \leq c\int \vert p(z) \vert^2 d\mu_2.
$$

 \item[(ii)] $\lambda(\mu_1,\mu_2)>0$ if and only if there is a constant $c>0$ such that for every polynomial $p(z)$
$$
\int \vert p(z) \vert^2 d\mu_1 \geq  c\int \vert p(z) \vert^2 d\mu_2.
$$

\end{itemize}
\end{prop}

\begin{proof}

\noindent Assume $\beta(\mu_{1}, \mu_{2})<  \infty$ is true, then for every polynomial  $p(z)=\sum_{k=0}^{n}v_kz^k$  and $v=(v_0,v_1,\dots,v_n,0,0,0,\dots)\in c_{00}$, with  $n\in \NN$,   via the identification (\ref{identification}) we have
that 
$$
v\mathbf{M}_iv^{*}=\int \vert p(z) \vert^2 d\mu_i, \qquad i=1,2.
$$

\noindent and therefore
$$
\beta(\mu_1,\mu_2)= \sup \{ \;  \dfrac{v\mathbf{M}_1v^{*}}{v\mathbf{M}_2v^{*}}, v\in c_{00}\setminus \{0\} \; \}.
$$
\noindent By taking  $c= \beta(\mu_1,\mu_2)>0$, it follows that for every  $p(z)\in \mathbb{P}[z]$

$$
\dfrac{\int \vert p(z) \vert^2 d\mu_1}{\int \vert p(z) \vert^2 d\mu_2} \leq c.
$$

\noindent Then, for every $p(z)\in \mathbb{P}[z]$,

$$
\int \vert p(z) \vert^2 d\mu_1 \leq c \int \vert p(z) \vert^2 d\mu_2,
$$

\noindent  On the other hand, assume  now that there exists 
 $c>0$ verifying that for  every $p(z)\in \mathbb{P}[z]$,
$$
\int \vert p(z) \vert^2 d\mu_1 \leq c \int \vert p(z) \vert^2 d\mu_2.
$$

\noindent By using again the identification (\ref{identification}) we have

\begin{eqnarray*}
\beta(\mu_1,\mu_2)&=& \sup \{ v\mathbf{M}_1v^{*} : v\mathbf{M}_2v^{*}=1,v\in c_{00}\} =\\
\mbox{} & = & 
\sup \{\int \vert p(z)\vert^2 d\mu_1 : \int \vert p(z) \vert^2 d\mu_2=1, p(z)\in \mathbb{P}[z]\} \leq c;
\end{eqnarray*}
\noindent  consequently $\beta(\mu_1,\mu_2)<\infty$ as we required.

\bigskip
\noindent  The proof for the index $\lambda$ is analogous.

\end{proof}

\bigskip

\section{Impact of generalized eigenvalues in the support of the measures}

\noindent First of all we begin with an easy result that shows  the impact of the behaviour of the diagonal sequence $\{c_{n,n}\}_{n=0}^{\infty}$ of a certain moment matrix $\mathbf{M}=(c_{i,j})_{i,j=0}^{\infty}$ in the boundedness of the support of the representing measure $\mu$. Note that the sequence $\{c_{n,n}\}_{n=0}^{\infty}$ verifies   the property:
\begin{equation} \label{C-S}
c_{n,n}^2 \leq c_{n-1,n-1}c_{n+1,n+1} \qquad n\geq 1
\end{equation}

\noindent Indeed, this is a consequence of  the Cauchy-Schwartz inequality of the inner product in the  Hilbert space  $L^{2}(\mu)$. In this case,  $\{\frac{c_{n,n}}{c_{n+1,n+1}}\}_{n=0}^{\infty}$ is a non decreasing sequence and consequently, by elementary results  
either
$\displaystyle{\lim_{n\to \infty} \vert c_{n,n}\vert^{\frac{1}{n}}}<\infty$ or $\displaystyle{\lim_{n\to \infty} \vert c_{n,n}\vert^{\frac{1}{n}}}=\infty$. The following  characterization of  boundedness of the support of the measure  is proved in \cite{Tomeo}. We prove it for the sake of completeness.

\begin{prop} \cite{Tomeo} \label{proposicion2}
Let $\mathbf{M}=(c_{i,j})_{i,j=0}^{\infty} $ be a positive definite moment matrix and let $\mu$ be any representing measure. Then,  the following are equivalent:
\begin{enumerate}
\item[i)] $\sop(\mu)$ is compact.
\item[ii)]  $\displaystyle{\lim_{n\to \infty}   \vert c_{n,n}\vert^{\frac{1}{2n}} < \infty}$
\end{enumerate}
\noindent Moreover, in this case if $R=\displaystyle{\lim_{n\to \infty} 
 \vert c_{n,n}\vert^{\frac{1}{2n}}}$ then $R= \sup \{ \vert z \vert: \;  z\in \sop(\mu)\}$.
\end{prop}\begin{proof} Assume first  that   $\sop(\mu)$ is a bounded set and $R= \sup \{ \vert z \vert: \; z\in \sop(\mu)\}$ (note that $R>0$ since the support is infinite). Then $\sop(\mu) \subset \overline{\mathbb{D}}_{R}(0)$ and

$$
\vert c_{n,n} \vert=\int \vert z \vert^{2n} \; d\mu \leq R^{2n} \mu\left(\overline{\mathbb{D}}_{R}(0)\right)=c_{00}R^{2n}
$$
\noindent Therefore, $\displaystyle{\lim_{n\to \infty}   \vert c_{n,n}\vert^{\frac{1}{2n}}\leq R <\infty}$.

\noindent On the other hand, assume that $\displaystyle{\lim_{n\to \infty}   \vert c_{n,n}\vert^{\frac{1}{2n}} =R}$, we prove that  $\sop(\mu)\subset \overline{\mathbb{D}}_{R}(0)$. Assume the contrary, then there exists $z_0\in \sop(\mu)$ such that $\vert z_0 \vert>R$. Thus thus we may choose $r>0$ and $\epsilon>0$ such that
$\vert z \vert \geq R+\epsilon$ for every $z\in \mathbb{D}_r(z_0)$. Therefore, for all $n\in \NN$
$$
\int \vert z \vert^{2n} \; d\mu \geq \int_{\mathbb{D}_r(z_0)} \vert z \vert^{2n} \; d\mu \geq (R+\epsilon)^{2n}\mu(\mathbb{D}_r(z_0))=c(R+\epsilon)^{2n}.
$$

\noindent Then, $\displaystyle{\lim_{n\to \infty}\vert c_{n,n}\vert^{\frac{1}{2n}} \geq R+\epsilon}$ which is not possible.
\end{proof}

\noindent As a consequence of the above result we obtain the following result:

\noindent
\begin{corollary} \label{corolario1}
Let $\mu$ be a measure such that $\beta(\mu,{\bf m}_r)<\infty$ for some $r>0$. Then, $\displaystyle{\lim_{n\to \infty} \vert c_{n,n} \vert^{\frac{1}{2n}}\leq r}$ and consequently $\sop(\mu) \subset \overline{\mathbb{D}}_{r}(0)$.
\end{corollary}
\begin{proof} Assume $\beta(\mu,{\bf m}_r)<\infty$, by Proposition \ref{proposicion1} there exists a constant $c>0$ verifying that for every polynomial, and in particular for polynomials  $z^n$ with $n\in \NN_{0}$ 

$$
c_{n,n} =
\int \vert z \vert^{2n} \; d\mu \leq c \int \vert z \vert^{2n} \; d{\bf m}_{r}= c \; r^{2n}.
$$

\noindent Therefore $\displaystyle{\lim_{n\to \infty}    \vert c_{n,n}\vert^{\frac{1}{2n}} \leq r}$ and  by   Proposition \ref{proposicion2} $\sop(\mu) \subset \overline{\mathbb{D}}_{r}(0)$.

\end{proof}

\begin{remark} Note that  $\displaystyle{\lim_{n\to \infty} 
  \vert c_{n,n}\vert^{\frac{1}{2n}} =r}$  is a weaker condition than $\beta(\mu,{\bf m}_{r})<\infty$. Indeed, consider a measure with support in $\mathbb{T}$ given by $d\mu(\theta)=w(\theta)d\theta$ with $w(\theta) \in L^{1}(-\pi,\pi)$   with $\essupp w(\theta)=\infty$ being $\essupp w(\theta)$ the essential upper bound, that is, the smallest number $C>0$ such that $w(\theta)\leq C$ a.e. By proposition $2$ it follows that  $\displaystyle{\lim_{n\to \infty} 
  \vert c_{n,n}\vert^{\frac{1}{2n}} =1}$. Nevertheless, $\beta(\mu,{\bf m})=\infty$; indeed,  by the results in  \cite{Grenader-Szego}  $\displaystyle{\lim_{n\to \infty} \beta_n=\essupp
  w(\theta)=\infty}$ being   $\beta_n$  the largest eigenvalue of the truncated matrix of $\mathbf{M}(\mu)$ of size $(n+1)\times (n+1)$.

\end{remark}

\noindent As it is established in Corollary \ref{corolario1}  if
$\beta(\mu , {\bf m}_{r})<\infty$  then $\sop(\mu) \subset \overline{\mathbb{D}}_{r}(0)$ which is not the support of ${\bf m}_r$ although it is, in particular, its convex envelope.  This result motivates the following problem:
Consider two compactly supported  $\mu_1,\mu_2$  measures verifying  that 
$\beta(\mu_1,\mu_2)<\infty $, we wonder if there is a relationship between the corresponding supports of $\mu_1,\mu_2$. Of course, it is not true in general that   $\sop(\mu_1) \subset \sop(\mu_2)$. Indeed, consider the Lebesgue measures  ${\bf m}$  and ${\bf m}_{\frac{1}{2}}$ in Example 2 verifying $\beta({\bf m}_{\frac{1}{2}}, {\bf m})=1$, nevertheless  $\sop({\bf m}_{\frac{1}{2}})= \frac{1}{2}\mathbb{T} $.

\noindent

\bigskip

\noindent As an answer of the above problem our  main  result in this section establishes the following 
relationship: if   $\beta(\mu_1,\mu_2)<\infty$ then  the support of $\mu_1$ must be contained in 
the {\it polynomially convex hull} of the  support of $\mu_2$.

\bigskip

\noindent We recall that ( see e.g.   \cite{Conwayfuncional})  
for a compact set  $K$ in the complex plane the {\it polynomially convex 
hull } of $K$, denoted $\Pc{(K)}$, is defined as:
$$
\Pc(K):=\{ z\in \CC:  \; \vert p(z) \vert \leq \max_{\xi \in K}\vert p(\xi) \vert,  \; \text{for all } p(z)\in \mathbb{P}[z] \}.
$$

\noindent A compact set is said to be {\it polynomially convex} if $K=\Pc(K)$. Obviously, $K \subset \Pc(K)$. In the complex plane polynomial convexity turns to be a purely topological notion. Indeed, using the maximum modulus principle and the Runge approximation theorem, one proves that a compact set is polynomially convex if and only if $\CC \setminus K$ is connected 

\begin{theorem} \label{teorema1}Let $\mu_1,\mu_2$ be compactly supported measures verifying  $\beta(\mu_1,\mu_2)<\infty$. Then, the support of $\mu_1$ is contained in the {\it polynomially convex hull} of the support of $\mu_2$, i.e., 
$$
\sop(\mu_1) \subset \Pc(\sop(\mu_2)).
$$
\end{theorem}

\begin{proof} Denote by $K_i={\it supp}(\mu_i)$, $i=1,2$. By \ref{proposicion1} it follows that 
 for every polynomial $p(z)$ it follows that
\begin{equation} \label{pol}
\int \vert p(z) \vert^2 d\mu_{1}  \leq  \int  \vert p(z) \vert^2 d\mu_{2} .
\end{equation}

\noindent Assume that  $K_1$ is not a subset of $\Pc(K_2)$. Then, there exists $z_1\in K_1$ such that $z_1 \notin \Pc(K_2)$. Since $ \Pc(K_{2})$ is polynomially convex from the results in  \cite{Forstneric} it follows that  there is a polynomial $p_0(z)\in \mathbb{P}[z]$ such that $p_0(z_1)=1$ and $ \vert p_0(z) \vert <1$ for every $z\in \Pc(K_{2})$. By compactness of $\Pc(K_{2})$  there exists $0 \leq \alpha <1$ such that
$\vert p_0(z) \vert \leq \alpha$ for all $z\in \Pc(K_{2})$. Next, since  $p(z_1)=1$ there exists $r>0$ such that  $\vert p_0(z) \vert > \frac{1+\alpha}{2}$ for all $z\in B_{r}(z_1)$. On the other hand, since  $z_1 \in K_1$ it follows that  $\mu_1(B_{r}(z_1))>0$. Therefore, by applying inequality (\ref{pol}) to the polynomials  $p_0^n(z)$ with  $n\in \NN$ we obtain that
$$
\left(\frac{1+\alpha}{2}\right)^n   \mu_1(B_{r}(z_1)) \leq \int_{K_1} \vert p_0(z)\vert^n d\mu_1 \leq \int_{K_2}  \vert p_0(z)\vert^n d\mu_2 \leq \alpha^{n}\mu_2(K_2).
$$

\noindent That means that there exists $C>0$ such that
$
\left(\frac{(1+\alpha)/2}{\alpha}\right)^n \leq C
$ for every  $n\in \NN$,
which is not possible since  $\frac{1+\alpha}{2}>\alpha$.
\end{proof}

\bigskip

\begin{remark}  The converse of this result is not true. Indeed, consider $\mu_1={\bf m}$ and 
the measure in the unit circle  $\mu_2
(\theta)=(1+\cos \theta)d\theta$ which apperars in  \cite{EGT1}. It is clear that
 $\sop(\mu_1) \subset \Pc(\sop(\mu_2))$ and nevertheless  $\beta(\mu_1,\mu_2)=\infty$. Indeed, by \cite{EGT1}  if $\mathbf{T}$ is the Toeplitz matrix associated with 
$\mu_2$ then 
$\lambda(\mu_2,\mu_1)=\displaystyle{\lim_{n\to 
\infty} \lambda_n(\mathbf{T}_n)}=0$ and therefore $\beta(\mu_1,\mu_2)=\infty$ by Remark 2.

\end{remark}

\noindent As a consequence of theorem \ref{teorema1} and remark 2  we have the following result

\begin{corollary} Let $\mu_1,\mu_2$ two compactly supported measures with infinite support with  polynomially convex supports.  If   $0<\lambda(\mu_1,\mu_2)\leq \beta(\mu_1,\mu_2)<\infty$ then the supports of both measures coincide.
\end{corollary}

\noindent It is interesting to point out that this result provides a tool to recognize the support of a measure comparing this measure, with our indexes, with another well known measure.  In particular, we obtain an application to find Jordan curves  in the support of a measure:

\begin{corollary} Let  $\Gamma$  be a Jordan curve and $\nu_{\Gamma}$ any measure  such that $\sop(\nu_{\Gamma})=\Gamma$. 
 Let $\mu$ be other measure such that verifies
$$
0< \lambda(\mu,\nu_{\Gamma}) \leq \beta(\mu,\nu_{\Gamma})<\infty.
$$
\noindent Then,
 $\Gamma \subset \sop(\mu) \subset \Pc{(\Gamma)}.$
 \end{corollary}

\smallskip

\begin{proof} Since $\beta(\mu,\nu_{\Gamma})<\infty$ by  Theorem 1  it follows that $\sop(\mu) \subset \Pc(\nu_{\Gamma})= \mathring{\Gamma}\cup \Gamma$. On the other hand by Remark 2 it follows that  $\beta(\nu_{\Gamma},\mu)=\dfrac{1}{\lambda(\mu,\nu_{\Gamma})}<\infty$, and again by using Theorem 1  
$$
\Gamma  = \sop(\nu_{\Gamma}) \subset \Pc(\sop(\mu)).
$$

\noindent Assume that $\Gamma$ is not contained in $\sop(\mu)$. Then there existe $z_1\in \Gamma$ such that $z_1\notin \sop(\mu)$. Then, 
$\sop(\mu)\subset \Gamma \setminus \mathbb{D}_s(z_1)$ for some $s>0$. Then, by  the definition of the polinomially convex hull 

$$\Gamma \subset \Pc(\sop(\mu))\subset \Pc\left((\mathring{\Gamma}\cup\Gamma) \setminus \mathbb{D}_s(z_1)\right)=
(\mathring{\Gamma}\cup\Gamma) \setminus \mathbb{D}_s(z_1)
$$

\noindent which is not possible. 
\end{proof}

\section{An algebraical approach of the convex envelope of a measure}

\noindent In the sequel we are concerned with the convex hull of the support  of a measure. 
 
We obtain  an algebraic description  of such set  in terms of Rayleigh quotients of a certain matrix (not necessarily Hermitian) related  with the associated moment matrix. In order to obtain the main result in this section  we need some definitions and lemmas.

\begin{definition}  Let $T$ be a bounded operator in a Hilbert space $\mathcal{H}$. 
The {\sl numerical range} of  $T$ is defined as
 $$
 W(T)=\{<Tx,x>: \;  x\in \mathcal{H}, \Vert x \Vert=1\}$$ 
\end{definition}

\begin{lem} Let $T$ be a bounded operator in a Hilbert space $\mathcal{H}$.  Assume that  $Y$ is dense set in $\mathcal {H}$,
then
\[ \overline{W}(T) =\overline{ \{\langle T y, y\rangle: \;  y \in Y, \|y\|=1\} }.\]
\end{lem}

\begin{proof}
Obviously
\[ \overline{ \{\langle T y, y\rangle: \;  y \in Y, \|y\|=1\}} \subset \overline{W}(T).\]
On the other hand, assume $z_{0} \in \overline{W}(T)$, then there is a sequence $\{x_{n}\}_{n=1}^{\infty}$ with $\Vert x_n \Vert=1$, such that $z_{0}=\displaystyle{ \lim_{n\to \infty} \langle Tx_{n}, x_{n} \rangle}$. Since $Y$ is a dense set in $\mathcal {H}$, there is a sequence
$\{y_{n}\}_{n=1}^{\infty}$ such that $\displaystyle{\lim_{n \to \infty} \|x_{n}-y_{n}\|=0}$; in particular, $\{y_n\}_{n=1}^{\infty}$ is a bounded sequence. Therefore,
\begin{multline*}
|\langle Tx_{n},x_{n} \rangle - \langle Ty_{n},y_{n}\rangle | = |\langle Tx_{n},x_{n}-y_{n} \rangle +
\langle Tx_{n}-Ty_{n},y_{n} \rangle |\leq \\
\leq \|T\| \; \|x_{n}\| \; \|x_{n}-y_{n}\|+\|T\|\; \|x_{n}-y_{n}\| \; \|y_{n}\| \to 0 \mbox{ if } n \to \infty.
\end{multline*}
Consequently, $z_{0}= \displaystyle{\lim_{n \to \infty} \langle Ty_{n}, 
y_{n} \rangle} \in  \overline{ \{\langle T y, y\rangle: \;  y \in 
Y, \|y\|=1\}}$.
\end{proof}
\bigskip

\begin{theorem}
Let $\mathbf{M}=(c_{i,j})_{i,j=0}^{\infty}$ be a moment matrix with $\displaystyle{\lim_{n\to \infty}   \vert c_{n,n}\vert^{\frac{1}{2n}} < \infty}$ and let $\mu$ be the representing   measure. Let ${\bf M}^{\prime}$ be the infinite matrix obtained by removing the first row in  ${\bf M}$, then 

\begin{equation} \label{1}
 \overline{\left \{ \frac{v {\bf M}^{\prime}v^{*}}{v {\bf M}
v^{*}} : v \in c_{00} \setminus \{0\}\right\}} =Co(\sup(\mu)).
\end{equation}
\end{theorem}

\begin{proof}
Let's see first in (\ref{1}) the inclusion from left to right. We will call $\Omega \subset \CC$ the set of complex numbers of the form
\[ \Omega = \left \{  \frac{v {\bf M}^{\prime}v^{*}}{v {\bf M}
v^{*}}: v \in c_{00} \setminus \{0\}\right\}\].

\noindent By using (\ref{identification}), if $p(z)=v_0+v_1z+\dots+v_nz^n$, and $v=(v_0,v_1,\dots,0,\dots)\in c_{00}$,

$$
\int z\vert p(z)\vert^{2}\; d\mu=
\int zp(z)\overline{p(z)}\; d\mu = v\mathbf{S}_{L}
\mathbf{M}v^{*}=v\mathbf{M}'v^{*}
$$

\noindent where $\mathbf{S}_{L}$ is the shift left matrix and $\mathbf{M}'=\mathbf{S}_{L}
\mathbf{M}$ is the infinite matrix obtained by removing the first row in $\mathbf{M}$. Therefore, $ \Omega$ can be rewritten as

\[ \Omega = \left \{ \frac{\int z |q(z)|^2 d\mu}{\int |q(z)|^2 d\mu}: q(z) \in \mathbb{P}[z] \setminus \{0\} \right \}.  \]
Let be $z_{0} \in \Omega$, by definition there exists  a polynomial $q(z) \in \mathbb{P}[z] $ such that
\[ z_{0} = \frac{\int z |q(z)|^2 d\mu}{\int |q(z)|^2 d\mu}.\]
This means that $z_{0}$ can be interpreted as the center of gravity of the weight function $w_{q}(z)=|q(z)|^2 d\mu$. In 
particular, if we consider the associated moment matrix with respect to  such measure denoted by  $\mathbf{M}(w_{q})=(\widetilde{c}_{i,j})_{i,j=0}^{\infty}$ it follows that $z_{0}$ is the zero of the  orthogonal polynomial of degree one  associated to the measure $w_{q}$ which is 
\[ \det \left (
\begin{array}{cc}
\widetilde{c}_{00} & 1 \\
\widetilde{c}_{10} & z
\end{array}
\right ) = z \widetilde{c}_{00}- \widetilde{c}_{10} = 0,\]
\noindent being
\[ z_{0} = \frac{\widetilde{c}_{10}}{\widetilde{c}_{00}}= \frac{\int z |q(z)|^2 d\mu}{\int |q(z)|^2 d\mu}.\]
As it is well known by Fejer's Theorem  ( see e.g. \cite{gaier}) 
  the zeros of the orthogonal polynomials are included in the convex envelope of the support
of the measure, and this proves the inclusion $\overline{\Omega} \subset \Co(\sop(\mu))$.

\noindent It remains to prove the reverse inclusion. Consider the multiplication by 
$z$-operator $\Delta_{z}:L^{2}(\mu) \to
L^{2}(\mu)$, given by $\Delta_{z}(p(z))=zp(z)$. This is a normal operator and it is well known (see e.g. \cite{Conway}) that the spectrum of this operator is the support of the measure $\mu$; i.e. $\sigma(\Delta_{z})= \sop(\mu)$.

\noindent Consider now the restriction operator of  $\Delta_{z}$ to  $P^{2}(\mu)$, $\mathcal {D}: P^{2}(\mu) \to
P^{2}(\mu)$; this is a subnormal operator which minimal normal extension 
is $\Delta_{z}$. By one of the basic properties relating the spectrum of a subnormal operator (see e.g.  \cite{Conway}), the spectrum of the normal
extension of an operator is contained in the spectrum of the operator, consequently
\[ \sop(\mu) \subset \sigma(\Delta_{z}) \subset \sigma(\mathcal {D}).\]

\noindent Now we use that  spectrum of an operator lies in the closure of its numerical range, that is  $\sigma(\mathcal {D}) \subset \overline{W}(\mathcal {D})$. On the other hand, the famous Toeplitz-Hausdorff theorem  (see e.g.  \cite{halmos}) 
 asserts that the numerical range of a bounded operator is always a convex set (no necessarily closed). Thus, the convex hull of the spectrum of an operator lies 
in the closure of the numerical range, therefore

\[ \Co(\sop(\mu)) \subset \Co(\sigma(\mathcal {D}) )\subset \overline{W}(\mathcal {D}).\]
Then, by lemma above, since $\mathbb{P}[z]$ is dense in the Hilbert space 
$P^{2}(\mu)$ it follows that
\begin{multline*}
\qquad \qquad  \overline{W}(\mathcal {D}) = \overline{\{ \langle \mathcal {D}(p),p \rangle :\;  \langle p, p\rangle=1 , p \in P^{2}(\mu)\}}= \\ =\overline{ \{ \langle zp,p \rangle: \;  p \in \mathbb{P}[z], \langle p,p\rangle =1\}},\qquad \qquad
 \end{multline*}
 and again via the identification (\ref{identification}) it follows
 \[\overline{W}(\mathcal {D}) = \overline{ \left \{ \frac{v {\bf M^{\prime}} v^{*}}{v{\bf M} v^{*}}: v \in c_{00} \setminus \{0\} \right \}}.\]
Consequently
\[ \Co(\sop(\mu)) \subset \overline{W}(\mathcal {D})= \overline{ \left \{ \frac{v {\bf M^{\prime}} v^{*}}{v{\bf M}v^{*}}: v \in c_{00} \setminus \{0\} \right \}}.\]
 as we required.
\end{proof}

\bigskip

\noindent  Given a certain compactly supported measure $\mu$ the matrix representation of the operator $\mathcal{D}$,  introduced in the proof of Theorem 2,  in the space $P^{2}(\mu)$ with respect to the associated basis of orthonormal polynomials is an upper Hessemberg matrix  (see e.g. \cite{EGT-hessemberg}). The following result provides a way to obtain the convex envelope of the support of a measure via such upper Hessemberg  matrix

\begin{corollary}
Let $\mu$ be a  compactly supported measure and let $\mathbf{D}$ be 
the Hessenberg matrix associated to $\mu$. Then:
$$
\Co(\sop(\mu))
=
 \overline{ \{v {\bf D} v^{*}:  v \in c_{00}, vv^{*}=1\}}= \overline{\bigcup_{n=0}^{\infty} W(\mathbf{D_n})}
 $$
\noindent where $\mathbf{D}_n$ is the $n+1$-section of the matrix $\mathbf{D}$.
\end{corollary}

\begin{ejem}
As it is known the identity matrix is the moment matrix associated with ${\bf m}$. 
 On the other hand
the matrix representation of the multiplication operator by $ z $, in the canonical basis of $\ell^2$, is
\[
 {\bf D}= \left ( \begin{array}{cccc}
0 & 0 & 0 & \ldots \\
1 & 0 & 0 & \ldots \\
0 & 1 & 0 & \ldots \\
\vdots & \vdots & \vdots & \ddots
\end{array}
\right ).\]
 We will get the value of
 $W({\bf D}_{n})$.
The projection on $\RR$
of field of values  $W({\bf D}_{n})$ is
\[ \Re({\bf D}_{n})=\frac{{\bf D}_{n}+{\bf D}^{*}_{n}}{2}=\left (
\begin{array}{cccccc}
0 & \frac{1}{2} & 0 & \ldots  &0 & 0 \\
\frac{1}{2} & 0 & \frac{1}{2} & \ldots & 0 & 0  \\
0 & \frac{1}{2} & 0 & \ldots & 0 & 0 \\
\vdots & \vdots & \vdots & \mbox{} & \vdots & \vdots \\
0 & 0 & 0& \ldots  & 0 & \frac{1}{2} \\
0 & 0 & 0  & \ldots  & \frac{1}{2} & 0 \\
\end{array}
\right ).\]
It is well known  that
 $\Re({\bf D}_{n})$ is the $n \times n$  tridiagonal
 secction
of $U_{n}(x)$, i.e. the Jacobi matrix  of orthogonal Tchebyschev polynomials of the second class on
 $[-1,1]$. We have that the
 eigenvalues of $\Re({\bf D}_{n})$ are the zeros of
  $U_{n}(x)$. It is well known that
  $U_{n}(x) = \sin((n+1) \arccos(x))/\sin( n \arccos(x))$.
  We set $n \in \NN$, and obtain
\begin{multline*}
\qquad \lambda_{\max}(\Re({\bf D}_{n})) = \max \{ U_{n}(x)=0 \} = \\
= \max \{\sin((n+1)\arccos(x)) =0\}= \max \{ \cos(\frac{\pi}{n+1}) \}. \qquad
\end{multline*}
  $\Re({\bf D}_{n})$ is a symmetric matrix hence
  $W(\Re({\bf D}_{n}))=[-\cos (\frac{\pi}{n+1}), \cos(\frac{\pi}{n+1})]$,
  and by the symmetry with respect to the origin we finally have
\[W({\bf D}_{n})=\{z \in \mathbb{C} : |z|\leq  \cos(\frac{\pi}{n+1})\}. \]
Therefore
\[ \overline{\bigcup_{n=1}^{\infty} W(\bf{D}_{n})} = \{ z \in \mathbb{C} : |z| \leq 1\},\]
and the corollary is fulfilled
since the convex envelope of support  of the Lebesgue measure
in the unit circumference,
 $\sop({\bf m})=\mathbb{T}$
 and $\Co( \sop({\bf m}))= \overline{\mathbb{D}}_1(0)$.

\end{ejem}

\noindent As a consequence of the above results we obtain the following identity that verifies  infinite moment matrices and which provides a necessary condition to be a moment matrix:

\begin{corollary} Let $\mathbf{M}=(c_{i,j})_{i,j=0}^{\infty}$ be a positive definite moment matrix and let  $\widehat{\mathbf{M}}$ be the matrix obtained by removing the first row and the first column. Then, 
$$
\beta(\widehat{\mathbf{M}},\mathbf{M})=\lim_{n\to \infty} 
\vert c_{n,n} \vert^{1/n}
$$
\end{corollary}
\begin{proof} First of all we show that the matrix $\widehat{\mathbf{M}} $ is also a moment matrix. Indeed, if $\mu$ is a representing measure for $\mathbf{M}$ and we compute  the moments of the measure  $\vert z \vert^{2}\mu$  we have  

$$
c_{i,j}(\vert z \vert^{2}\mu)= \int \vert z \vert^2z^{i} \overline{z}^{j}d\mu=
\int z^{i+1}\overline{z}^{j+1}d\mu .
$$

\noindent This means that  the matrix obtained by removing the first row and column in  $\mathbf{M}$, that we denote
$\widehat{\mathbf{M}}$, is the moment matrix associated with the measure $\vert z \vert^{2}\mu$. Therefore, if $\Delta_{z}:P^{2}(\mu) \to P^{2}(\mu)$ is the multiplication by-$z$-operator, 
\begin{eqnarray*}
\Vert \Delta_{z} \Vert^{2} & = & \sup \{\int z p(z) \overline{zp(z)} d\mu  : \int  p(z) \overline{p(z)}  d\mu =1, p\in \mathbb{P}[z]\}=\\
\mbox{} & = & 
\sup \{ \int \vert p(z) \vert^{2} \vert z \vert^{2} d\mu : \int \vert p(z) \vert^{2}   d\mu =1, p\in \mathbb{P}[z]\}.
\end{eqnarray*}

\noindent By using  (\ref{identification}) we have
\begin{eqnarray*}
\Vert \Delta_{z} \Vert^{2}  & = & 
\sup \{ v\widehat{\mathbf{M}}v^{*}: v\in c_{00}, v\mathbf{M}v^{*}=1\} \\
\mbox{} &
= & 
\beta(\vert z \vert^{2} d\mu, d\mu)=\sup \{\vert z \vert^2: \;  z\in \sop(\mu)\} = \beta(\widehat{\mathbf{M}},\mathbf{M}).
\end{eqnarray*}

\noindent Now the result follows from Proposition $1$. 
\end{proof}

\begin{remark}
\noindent  Note that for an  HPD Toeplitz matrix $\mathbf{T}$ it is obvious that  $ \beta (\widehat {\mathbf{T}}, \mathbf{T}) = 1 $. Nevertheless, this identity  is not true in general for infinite HPD matrices; indeed, consider  the infinite diagonal matrix
$\mathbf{M}={\sl diag }(1,2,1,2^2,1,2^3,1,2^4, \ldots)$ which is an  HPD matrix such that   $\displaystyle{\lim_{n\to \infty} 
\vert c_{n,n} \vert^{1/n}}$ does not exist. Note that $\mathbf{M}$  is not a moment matrix since the elements of the diagonal  does not fullfile  Cauchy-Schwarz inequality  since $(2^n)^2 \not \leq 1 \cdot 1$.
\end{remark}

\noindent 
\section{A generalization of a sufficient condition for completeness of polynomials.}

\noindent In \cite{EGT1} a sufficient condition for completeness of polynomials is given in the following terms: if polynomials are dense in a certain measure space $L^{2}(\mu)$ for a compactly supported measure then the smallest eigenvalue of the truncated matrices of the moment matrices $\lambda_n\to 0$, as $n\to \infty$; using our notation $\lambda(\mu,{\bf m})=0$. There is proved that this is not  necessary condition for completeness of polynomials. Indeed, consider the measures ${\bf m}_{r}$ with $r<1$; by (\ref{ejemplo1})    $\lambda({\bf m}_{r},{\bf m})= 0$ and nevertheless polynomials are not dense in $L^{2}({\bf m}_{r})$. We here generalize this result comparing with the Lebesgue measures in circles with different center and radios. The proof is based in the following matrix version of the change of variable theorem:

\noindent Recall that if  $\varphi(z)=\alpha z + z_0$ is a similarity map  onto $\mathbb{C}$ with $\alpha,z_0\in \CC$  it is proved in \cite{EGT5} and  \cite{EST}  that  the moment matrix of the image measure $\mu \circ \varphi^{-1}$,  denoted by $\mathbf{M}^{\varphi}(\mu)$, has finite section  of size $(n+1)\times(n+1)$

\begin{equation} \label{semejanza}
 \mathbf{M}^{\varphi}_n(\mu) = {\bf A}_{n}(\alpha,\beta) {\bf M}_{n}(\mu){\bf A}_{n}^{*}
(\alpha,\beta )
\end{equation}
where ${\bf A}_{n}(\alpha, \beta)$ is defined
as in \cite{EST} ( with certain  modifications due to the different expression for the inner product)
\begin{equation}
{\bf A}_{n}(\alpha, \beta) =
\left (
\begin{array}{lllll}
{\binom 0 0} \alpha^{0} \beta^{0} & 0 & 0 & \ldots & 0 \\
{\binom 1 0} \alpha^{0}
\beta^{1} & {\binom 1 1} \alpha^{1} \beta^{0} & 0 & \ldots & 0 \\
{\binom 2 0} \alpha^{0} \beta^{2} &  {\binom 2 1} \alpha^{1}
\beta^{1}  & {\binom 2 2} \alpha^{2} \beta^{0} & \ldots & 0\\
\quad\vdots & \quad \vdots & \quad \vdots & \mbox{} & \quad \vdots \\
{\binom n 0}
\alpha^{0} \beta^{n} &
{\binom n 1} \alpha^{1} \beta^{n-1} & {\binom n 2}
\alpha^{2} \beta^{n-2} & \ldots & {\binom n n} \alpha^{n} \beta^{0}
\end{array}
\right ).
\end{equation}

\smallskip 

\noindent In the particular case of $\alpha =r\in \RR$ we obtain the folllowing result: 
\begin{lem}  \label{lema2}
Let $\varphi(z)=r z +z_0$ be a certain similarity map with $r>0$ and $\mu$ a measure. Consider the image measure $\mu\circ \varphi^{-1}$  obtained after similarity map. Then,  
$$
\lambda(\mu,{\bf m})=\lambda(  \mu\circ \varphi^{-1}  ,{\bf m}_{z_0; r}),\qquad \qquad \beta (\mu,{\bf m})=\beta ( \mu\circ \varphi^{-1}  ,{\bf m}_{z_0; r})
$$

\end{lem}
\begin{proof} 

We prove the result for the index $\lambda$ (for the other index the proof is analogous). Denote by $\mathbf{M}_{1}=\mathbf{M}(\mu)$ and $\mathbf{M}_{2}=\mathbf{M}^{\varphi}(\mu)$. We have seen in (\ref{semejanza}) that for all $n\in \NN_{0}$
$$
 \mathbf{M}^{\varphi}_n=\mathbf{A}_{n}(r;z_0) \mathbf{M}_{n} \mathbf{A}_{n}^{*}(r;z_0).
$$

\noindent Consequently for every $v\in \CC^{n+1}$ it follows

$$
v \mathbf{M}^{\varphi}_nv^{*}=v\mathbf{A}_{n}(r;z_0) \mathbf{M}_{n} \mathbf{A}_{n}^{*}(r;z_0)v^{*}
$$

\noindent and we have
\begin{eqnarray*}
\lambda_{n}(\mu\circ \varphi^{-1},{\bf m}_{z_0;r})&  = & \inf \{ v
\mathbf{M}^{\varphi}_{n} v^{*} : v \mathbf{M}_n({\bf m}_{z_0;r})
v^{*}=1\}\\
\mbox{} &  = & 
\inf \{ v\mathbf{A}_{n}(r;z_0) \mathbf{M}_{1,n} \mathbf{A}_{n}^{*}(r;z_0)
 v^{*} : v \mathbf{A}_{n}(r;z_{0}) \mathbf{I}_{n}
\mathbf{A}_{n}^{*}(r;z_{0})
v^{*}=1\}.
\end{eqnarray*}

\noindent  By making the change $w=v\mathbf{A}_{n}(r;z_0)$ in the last equality,
it follows
\[ \lambda_{n}(\mu\circ \varphi^{-1},{\bf m}_{z_0;r})= \inf (w \mathbf{M}_{n} w^{*}:
w \mathbf{I}_{n} w^{*}=1\},
\]
and by taking limits we obtain 
$$
\lambda(\mu\circ \varphi^{-1},{\bf m}_{z_0;r})= \inf \{ w \mathbf{M}(\mu) w^{*} :ww^{*}=1\} =\lambda(\mu,{\bf m})
$$

\noindent as we required.

\end{proof}

\begin{remark} The above lemma can be obviously extended  for any 
infinitely supported 
measures in the complex plane  $\mu_1,\mu_2$  in the following way: let   $\varphi$ be  a similarity map then 
$$
\lambda(\mu_1,\mu_2)=\lambda(\mu_1\circ \varphi^{-1}
,\mu_2\circ \varphi^{-1})\;, \qquad  \beta(\mu_1,\mu_2)=\beta(\mu_1\circ \varphi^{-1}
,\mu_2\circ \varphi^{-1}).
$$
\end{remark}

\bigskip

\begin{theorem} \label{teorema4}
Let $\mu$  be a  compactly supported measure.   If  $P^{2}(\mu)=L^{2}(\mu)$, then for every  $z_0\in \CC$ and $r>0$, it holds
$$
\lambda \left(\mu,{\bf m}_{z_0;r} \right )=0.
$$
\end{theorem}

\begin{proof} Assume the contrary. Then, there exists $z_0\in \CC$ and $r>0$ such that
$$
\lambda(\mu,{\bf m}_{z_0;r})>0.
$$
\noindent Then, consider the similarity $\varphi(z)=\frac{1}{r}z-z_0$ and the image measure $\mu \circ \varphi^{-1}$ obtained by applying the similarity $\varphi$ to $\mu$. Then, by lemma (\ref{lema2}) it follows that
$$
\lambda(\mu \circ \varphi^{-1},{\bf m})=\lambda(\mu,{\bf 
m}_{z_0;r})>0
$$

\noindent Then by the results in  \cite{EGT1} it follows that $P^{2}(\mu \circ \varphi^{-1}) \neq L^{2}(\mu \circ \varphi^{-1})$ and consequently
$L^{2}(\mu)\neq P^{2}(\mu)$.

\end{proof}

\noindent Above theorem can be rewritten as follows: if a   measure $\mu$ verifies that 
 $\lambda(\mu,{\bf m}_{z_0;r})>0 $ for a certain Lebesgue measure ${\bf m}_{z_0;r}$, obviously non  complete (i.e. polynomaials are not dense)   then also $\mu$ is not complete. One can wonder if we may replace measures ${\bf m}_{z_0;r}$
 by any non complete measure and this motivates the following problem: 
 
\smallskip

\noindent {\bf Problem:} Assume $\mu_1,\mu_2$ two compactly supported measures verifying: 
\begin{enumerate}
\item $\mu_2$ is not complete, i.e., $P^{2}(\mu_2)\neq L^{2}(\mu_2)$.
\item $\lambda(\mu_1,\mu_2)>0$.
\end{enumerate}
\noindent Is it true that $L^{2}(\mu_1)\neq P^{2}(\mu_1)$?

\bigskip

\noindent We  obtain a partial answer  when the first measure is supported in a Jordan curve:

\begin{prop} Let $\mu_1,\mu_2$ be infinitely supported measures in the complex plane such that  $P^{2}(\mu_2)\neq L^{2}(\mu_2)$ and  $\lambda(\mu_1,\mu_2)>0$. Then, the following holds: 
\smallskip 

\begin{center}
    If  $\mu_1$ is supported on a Jordan curve $\Gamma$, then  $P^{2}(\mu_1) \neq L^{2}(\mu_1)$.
\end{center}

\end{prop}

\begin{proof} By Thomson's  theorem \cite{Thomson} since  $L^{2}(\mu_2) \neq P^{2}(\mu_2)$ it follows the existence of a bounded point evaluation for $\mu_2$, say $\xi$. Recall that   $\xi\in \CC$ is a bounded point evaluation for the measure  $\mu_2$ if there exists a constant $c>0$ such that
 for every $p(z)\in \mathbb{P}[z]$,
$$
\vert p(\xi) \vert^2 \leq c \int \vert p(z)\vert^2 d\mu_2.
$$

\noindent Since $\lambda(\mu_1,\mu_2)>0$ it  easily follows that $\xi$ is a bounded point evaluation of $\mu_1$. Indeed, by Proposition \ref{proposicion1} there exists $C>0$ such that for every polynomial 
$p(z)$ 
$$
\int \vert p(z) \vert^2 d \mu_1 \geq C \int \vert p(z) \vert^2 d \mu_2.
$$

\noindent Consequently, combining both inequalities we have that for every polynomial $p(z)$ 
 $$
\vert p(\xi)\vert^2 \leq C\int \vert p(z) \vert^2 d\mu_2 \leq
\dfrac{c}{C} \int \vert p(z) \vert^2 d\mu_1
$$

\noindent and  $\xi$ is a bounded point evaluation for $\mu_1$. The conclusion now follows by  \cite{EGT5} since for supported in Jordan curves measures  the existence of a bounded point evaluation implies that the measure is not complete.

\end{proof}

\today

\end{document}